\documentclass{amsart}
\usepackage{color}
\usepackage{graphicx,epsf}
\usepackage{amsmath,amscd}
\newtheorem{theorem}{Theorem}[section]

\newtheorem{corollary}[theorem]{Corollary}

\theoremstyle{definition}
\newtheorem{definition}[theorem]{Definition}

\theoremstyle{remark}
\newtheorem{remark}[theorem]{Remark}
\numberwithin{equation}{section}

\newcommand{\Real}{{\mathbb R}}

\newcommand{\f}{\mathbf{f}}
\newcommand{\x}{\mathbf{x}}

\newcommand {\hide}[1]{}

\begin{document}
\title[Toric cubes are closed balls]
{Toric cubes are closed balls}
\author{Saugata Basu}
\address{Department of Mathematics,
Purdue University, West Lafayette, IN 47907, USA}
\email{sbasu@math.purdue.edu}
\author{Andrei Gabrielov}
\address{Department of Mathematics,
Purdue University, West Lafayette, IN 47907, USA}
\email{agabriel@math.purdue.edu}
\author{Nicolai Vorobjov}
\address{
Department of Computer Science, University of Bath, Bath
BA2 7AY, England, UK}
\email{nnv@cs.bath.ac.uk}
\thanks{The first author was supported in part by NSF grant CCF-0915954.
The second author was supported in part by NSF grants DMS-0801050 and DMS-1067886}

\begin{abstract}
We prove that toric cubes, 
which are images of $[0,1]^d$ under monomial maps, are the closures
of graphs of monotone
maps, and in particular
semi-algebraically homeomorphic to closed balls.
\end{abstract}
\maketitle

\section{Introduction}
In \cite{Sturmfelsetal2012} Engstr\"om, Hersh and Sturmfels introduced a class of compact
semi-algebraic sets which they call \emph{toric cubes}.

The following definition is adapted from \cite{Sturmfelsetal2012}.
\begin{definition}\label{def:toric}
Let $\mathcal{A} = \{\mathbf{a}_1,\ldots,\mathbf{a}_n\} \subset \mathbb{N}^d$,
and $f_{\mathcal{A}}: [0,1]^d \rightarrow [0,1]^n$ be the map
\[
\mathbf{t} = (t_1,\ldots,t_d) \mapsto (\mathbf{t}^{\mathbf{a}_1},\ldots,
\mathbf{t}^{\mathbf{a}_n}),
\]
where $\mathbf{t}^{\mathbf{a}_i}:= t_{1}^{a_{i,1}} \cdots t_{d}^{a_{i,d}}$ for
$\mathbf{a}_i=(a_{i,1}, \ldots ,a_{i,d})$.
The image of $f_\mathcal{A}$ is called a toric cube.

We call the image of the restriction of $f_\mathcal{A}$ to $(0,1)^d$ an \emph{open toric cube}.
The closure of an open toric cube is a toric cube.
Note that an open toric cube is not necessarily an open subset of $\Real^n$, and need not be
contained in  $(0,1)^n$ (if some $\mathbf{a}_i = \mathbf{0}$).
\end{definition}

In \cite{BGV,BGV2012} the authors introduced a certain
class of definable subsets of $\Real^n$ (called \emph{semi-monotone sets})
and definable maps $f: X \rightarrow \Real^k$ (called
\emph{monotone maps}), where $X \subset \Real^n$ is a semi-monotone set.
Here ``definable'' means ``definable in an o-minimal structure
over $\Real$'', for example, real semi-algebraic.

These objects are meant to serve as building blocks for obtaining
a conjectured cylindrical cell decomposition
of definable sets into topologically regular cells,
without changing the coordinate system in the ambient space $\Real^n$
(see \cite{BGV,BGV2012} for a more detailed motivation behind these
definitions).

The main result of this note is the following theorem.

\begin{theorem}
\label{thm:main}
An open toric cube $C \subset \Real^n$
is the graph of a monotone map.
\end{theorem}

As a result we obtain
\begin{corollary}
\label{cor:main}
An open toric cube $C \subset [0,1]^n$, with $\dim (C) = k$,
is semi-algebraically homeomorphic to a standard open ball.
The pair $(\overline{C},C)$ is semi-algebraically homeomorphic to the pair
$([0,1]^{k},(0,1)^{k})$, in particular, a toric cube is semi-algebraically homeomorphic to
a standard closed ball.
\end{corollary}

\begin{remark}
Note that the first statement in Corollary~\ref{cor:main} is also proved in \cite[Proposition 1]
{Sturmfelsetal2012}.
In conjunction with Theorem~2 in \cite{Sturmfelsetal2012},
Corollary~\ref{cor:main} implies that any CW-complex in which the closures of each cell is a toric cube,
must be a regular cell complex, and this answers in the affirmative the
Conjecture~1 in \cite{Sturmfelsetal2012}.
\end{remark}

\section{Proof of Theorem \ref{thm:main} and Corollary \ref{cor:main}}
We begin with a few preliminary definitions.

\begin{definition}
\label{def:semi-monotone}
Let $L_{j, \sigma, c}:= \{ \x=(x_1, \ldots ,x_n) \in \Real^n|\> x_j \sigma c \}$
for $j=1, \ldots ,n$, $\sigma \in \{ <,=,> \}$, and $c \in \Real$.
Each intersection of the kind
$$C:=L_{j_1, \sigma_1, c_1} \cap \cdots \cap L_{j_m, \sigma_m, c_m} \subset \Real^n,$$
where $m=0, \ldots ,n$, $1 \le j_1 < \cdots < j_m \le n$, $\sigma_1, \ldots ,\sigma_m \in \{<,=,> \}$,
and $c_1, \ldots ,c_m \in \Real$, is called a {\em coordinate cone} in $\Real^n$.

Each intersection of the kind
$$S:=L_{j_1, =, c_1} \cap \cdots \cap L_{j_m, =, c_m} \subset \Real^n,$$
where $m=0, \ldots ,n$, $1 \le j_1 < \cdots < j_m \le n$,
and $c_1, \ldots ,c_m \in \Real$, is called an {\em affine coordinate subspace} in $\Real^n$.

In particular, the space $\Real^n$ itself is both a coordinate cone and an affine coordinate
subspace in $\Real^n$.
\end{definition}

\begin{definition}[\cite{BGV}]
\label{def:set}
An open (possibly, empty) bounded set $X \subset \Real^n$ is called {\em semi-monotone} if
for each coordinate cone $C$  the intersection $X \cap C$ is connected.
\end{definition}

\begin{remark}
In fact, in Definition \ref{def:set} above, it suffices to consider
intersections with only affine
coordinate subspaces
(see \cite[Theorem~4.3]{BGV2012} or Theorem \ref{th:def_monotone_map} below).
\end{remark}

Notice that any convex open subset of $\Real^n$ is semi-monotone.

The definition of \emph{monotone maps} is given in \cite{BGV2012} and is a bit more technical.
We will not repeat it here but recall a few important properties of monotone maps that we will need.
In particular, Theorem~\ref{th:def_monotone_map} below, which appears in \cite{BGV2012},
gives a complete characterization of monotone maps.
For the purposes
of the present paper this characterization can be taken as the definition of monotone maps.

\begin{definition}
[\cite{BGV2012}, Definition~ 1.4]
\label{def:quasi-affine}
Let a bounded continuous map $\f=(f_1, \ldots ,f_k)$ defined on an open bounded non-empty set
$X \subset \Real^n$ have the graph ${\bf F} \subset \Real^{n+k}$.
We say that $\f$ is {\em quasi-affine} if for any coordinate subspace
$T$ of $\Real^{n+k}$, the projection $\rho_T:\> {\bf F} \to T$ is injective if and only if the image
$\rho_T({\bf F})$ is $n$-dimensional.
\end{definition}

The following 
theorem is
proved in \cite{BGV2012}.

\begin{theorem}[\cite{BGV2012}, Theorem~4.3]
\label{th:def_monotone_map}
Let a bounded continuous quasi-affine map $\f=(f_1, \ldots ,f_k)$ defined on an open bounded
non-empty set $X \subset \Real^n$ have the graph ${\bf F} \subset \Real^{n+k}$.
The following three statements are equivalent.

\begin{enumerate}
\item[(i)]
The map $\f$ is monotone.
\item[(ii)]
For each affine coordinate subspace $S$ in $\Real^{n+k}$ the intersection ${\bf F} \cap S$ is connected.
\item[(iii)]
For each coordinate cone $C$ in $\Real^{n+k}$ the intersection ${\bf F} \cap C$ is connected.
\end{enumerate}
\end{theorem}

\hide{
\begin{corollary}[\cite{BGV2012}, Corollary~4.4]\label{cor:def_monotone_map}
Let $\f:\> X \to \Real^k$ be a monotone map having the graph ${\bf F} \subset \Real^{n+k}$.
Then for every coordinate $z$ in $\Real^{n+k}$ and every $c \in \Real$,
each of the intersections ${\bf F} \cap \{ z\> \sigma\> c \}$, where $\sigma \in \{ <,>,= \}$,
is either empty or the graph of a monotone map.
\end{corollary}

\begin{theorem}[\cite{BGV2012}, Theorem~4.6]\label{th:proj}
Let $\f:\> X \to \Real^k$ be a monotone map defined on a semi-monotone set $X \subset \Real^n$ and
having the graph ${\bf F} \subset \Real^{n+k}$.
Then for any coordinate subspace $T$ in $\Real^{n+k}$
the image $\rho_T ({\bf F})$ under the projection map $\rho_T:\> {\bf F} \to T$
is either a semi-monotone set or the graph of a monotone map.
\end{theorem}
}

\begin{remark}
In view of Theorem~\ref{th:def_monotone_map}, it is natural to identify any semi-monotone set
$X \subset \Real^n$ with the graph of an identically constant function $f \equiv c$ on $X$,
where $c$ is an arbitrary real.
\end{remark}

\begin{definition}
A definable bounded open set $U \subset \Real^n$ is called (topologically) 
regular 
cell
if $\overline U$ is definably homeomorphic to a closed ball,
and the frontier $\overline U \setminus U$ is definably homeomorphic
$(n-1)$-sphere.
In other words, the pair $(\overline{U},U)$ is definably homeomorphic
to the pair $([0,1]^n,(0,1)^n)$.
\end{definition}

\begin{theorem}[\cite{BGV2012}, Theorem~5.1]
\label{th:regularcell}
The graph ${\bf F} \subset \Real^{n+k}$ of a monotone map $\f: X \to \Real^k$ on a semi-monotone set
$X \subset \Real^n$ is
definably homeomorphic to a regular 
cell.
\end{theorem}

\begin{proof}[Proof of Theorem \ref{thm:main}]
Let $C \subset [0,1]^n$ be an open toric cube and suppose that $C=f_\mathcal{A}((0,1)^d)$ for
a monomial map $f_\mathcal{A}$ (see Definition~\ref{def:toric}).

Make the coordinate change $z_i = \log(t_i)$ for every $i=1, \ldots, d$, and take the logarithm
of every component of the map $f_\mathcal{A}$ expressed in coordinates $z_i$.
Denote the resulting map by $\log f_\mathcal{A}$.
Then $\log f_\mathcal{A}$ is the restriction of a linear map, namely
\[
\log f_\mathcal{A}: (-\infty,0)^d \rightarrow (-\infty,0)^n,
\]
defined by
\[
\mathbf{z} = (z_1,\ldots,z_d) \mapsto (\mathbf{a}_1\cdot\mathbf{z},
\ldots,\mathbf{a}_n\cdot\mathbf{z}).
\]

Observe that $\log$ (the component-wise logarithm) maps the open cube, $(0,1)^d$ (resp.
$(0,1)^n$) homeomorphically onto $(-\infty,0)^d$ (resp. $(-\infty,0)^n$).
It follows  that the fiber of the orthogonal projection of $C$ to any $k$-dimensional
coordinate subspace is the pre-image under the $\log$ map
of an affine subset of $(-\infty,0)^n$, and is a single point if it is zero-dimensional.
Hence $C$ is a graph of a quasi-affine map (choose any set of $k$ coordinates such that
the image of $C$ under the orthogonal projection to the coordinate subspace of those
coordinates is full dimensional).

Similarly, the intersection of $C$ with any affine coordinate subspace is the
pre-image under the $\log$ map, of an affine subset of $(-\infty,0)^n$ and hence connected.

We proved that $C$ satisfies the conditions of Theorem~\ref{th:def_monotone_map}, hence
$C$ is the graph of a monotone map.
\end{proof}

\begin{proof}[Proof of Corollary \ref{cor:main}]
Immediate consequence of Theorem \ref{thm:main} and
Theorem \ref{th:regularcell}.
\end{proof}

\bibliographystyle{plain}
\bibliography{master}
\end{document}